\documentclass[a4paper,12pt]{amsart}

\usepackage{a4wide}
\usepackage{tikz}
\usepackage{url}

\newif\ifdetails
\detailstrue
\newcommand{\DETAIL}[1]%
{\ifdetails\par\fbox{\begin{minipage}{0.9\linewidth}\textit{Detail:}
      #1\end{minipage}}\par\fi}
\newcommand{\TODO}[1]%
{\ifdetails\par\fbox{\begin{minipage}{0.9\linewidth}\textbf{TODO:}
      #1\end{minipage}}\par\fi}

\usepackage{makecell}
\usepackage{relsize}

\newtheorem{lemma}{Lemma}
\newtheorem{proposition}[lemma]{Proposition}
\newtheorem{theorem}[lemma]{Theorem}
\newtheorem{corollary}[lemma]{Corollary}
\theoremstyle{remark}

\newtheorem{remark}{Remark}

\newtheorem{question}{Question}

\usepackage{caption}
\usepackage{subcaption}
\usepackage{float} 
\usepackage[pdftex,a4paper,
citecolor = blue, colorlinks=true,urlcolor=blue]{hyperref}
\newcommand{\old}[1]{{}}

\title{Inducibility of topological trees}

\author{Audace A. V. Dossou-Olory}
\author{Stephan Wagner}
\thanks{The first author was supported by Stellenbosch University and African Institute for Mathematical Science (AIMS) South Africa, the second author was supported by the National Research Foundation of South Africa, grant number 96236.}
\address{Audace A. V. Dossou-Olory and Stephan Wagner\\ Department of Mathematical Sciences \\ Stellenbosch University \\ Private Bag X1, Matieland 7602 \\ South Africa}
\email{\{audaced,swagner\}@sun.ac.za}

\begin{document}

\begin{abstract}
Trees without vertices of degree $2$ are sometimes named topological trees. In this work, we bring forward the study of the inducibility of (rooted) topological trees with a given number of leaves. The inducibility of a topological tree $S$ is the limit superior of the proportion of all subsets of leaves of $T$ that induce a copy of $S$ as the size of $T$ grows to infinity. In particular, this relaxes the degree-restriction for the existing notion of the inducibility in $d$-ary trees. We discuss some of the properties of this generalised concept and investigate its connection with the degree-restricted inducibility. In addition, we prove that stars and binary caterpillars are the only topological trees that have an inducibility of $1$. We also find an explicit lower bound on the limit inferior of the proportion of all subsets of leaves of $T$ that induce either a star or a binary caterpillar as the size of $T$ tends to infinity.
\end{abstract}

\maketitle

\section{Introduction and preliminary}

The parameter \textit{inducibility} of graphs first appears in the mathematical literature in a paper of Pippenger and Golumbic \cite{pippenger1975inducibility} in 1975. In plain language, Pippenger and Golumbic studied the maximum frequency of a simple graph occurring as a subgraph induced by vertices of another simple graph with a sufficiently large number of vertices. Over the past four decades, many authors have extensively discussed the inducibility in various settings and a number of papers has been published.

The inducibility of trees, however, was first initiated much more recently in \cite{bubeck2016local} by Bubeck and 
Linial. They studied the concept in trees with a given number of vertices. Shortly afterwards, Czabarka, Sz{\'e}kely and the second author of the current paper introduced in \cite{czabarka2016inducibility} the inducibility in rooted binary trees with a given number of leaves and also presented an application to random tanglegrams. The recent paper~\cite{AudaceStephanPaper1} proposed an extension of the inducibility of binary trees to $d$-ary trees for every $d\geq 2$.

\medskip
The object of this paper is to continue the investigation of the inducibility of trees. In particular, it is natural to consider a variant of the concept that relaxes the degree-restriction on the vertices of the tree: one might be interested in knowing how large the number of appearances of a tree in another larger tree with a given number of leaves can be.

Before we can present our results, we need to give some formal definitions.

\medskip
An unrooted tree that does not contain any subdivided edge, that is, no vertices of degree $2$, is called a \emph{topological} tree (see e.g. Bergeron et al. \cite{bergeron1998combinatorial} or Allman and Rhodes \cite{allman2004mathematical}), or also \emph{series-reduced} or \emph{homeomorphically irreducible} tree.
Based on this definition, we shall call any rooted tree in which every vertex has outdegree at least $2$ a (rooted) topological tree.

When the tree consists of a single vertex, we shall treat it as both leaf and root. Figure~\ref{fewer5leaves} displays all the topological trees with fewer than five leaves: the counting sequence for the number of $n$-leaf topological trees starts
\begin{align*}
1, 1, 2, 5, 12, 33, 90, 261, 766, 2312, 7068, \ldots\,,
\end{align*}
see A000669 in \text{\cite{oeis}} for more information.

\begin{figure}[htbp] \centering
  \begin{subfigure}[b]{0.16\textwidth} \centering 
  \begin{tikzpicture}
\filldraw [black] circle (3pt);
\end{tikzpicture}
  \caption{$1$}
  \end{subfigure}~
  \begin{subfigure}[b]{0.16\textwidth} \centering     
\begin{tikzpicture}[thick,level distance=9mm]
\tikzstyle{level 1}=[sibling distance=11mm]   
\node [circle,draw]{}
child {[fill] circle (2pt)}
child {[fill] circle (2pt)};
\end{tikzpicture}
  \caption{$2$}
  \end{subfigure}~
  \begin{subfigure}[b]{0.16\textwidth} \centering 
\begin{tikzpicture}[thick,level distance=9mm]
\tikzstyle{level 1}=[sibling distance=11mm]
\tikzstyle{level 2}=[sibling distance=9mm]
\node [circle,draw]{}
child {child {[fill] circle (2pt)} child {[fill] circle (2pt)}}
child {[fill] circle (2pt)};
\end{tikzpicture}
   \caption{$3$}
  \end{subfigure}~
   \begin{subfigure}[b]{0.18\textwidth} \centering 
\begin{tikzpicture}[thick,level distance=10mm]
\tikzstyle{level 1}=[sibling distance=8mm]
\node [circle,draw]{}
child {[fill] circle (2pt)}
child {[fill] circle (2pt)}
child {[fill] circle (2pt)};
\end{tikzpicture}
  \caption{$3$}
  \end{subfigure}~
  \begin{subfigure}[b]{0.18\textwidth} \centering
\begin{tikzpicture}[thick,level distance=10mm]
\tikzstyle{level 1}=[sibling distance=11mm]
\tikzstyle{level 2}=[sibling distance=7mm]
\node [circle,draw]{}
child {child {[fill] circle (2pt)}child {[fill] circle (2pt)}child {[fill] circle (2pt)}}
child {[fill] circle (2pt)};
\end{tikzpicture}
  \caption{$4$}
  \end{subfigure}\newline \newline
  \begin{subfigure}[b]{0.2\textwidth} \centering
\begin{tikzpicture}[thick,level distance=10mm]
\tikzstyle{level 1}=[sibling distance=12mm]
\tikzstyle{level 2}=[sibling distance=7mm]
\node [circle,draw]{}
child {child {[fill] circle (2pt)} child {[fill] circle (2pt)}}
child {child {[fill] circle (2pt)} child {[fill] circle (2pt)}};
\end{tikzpicture}
 \caption{$4$}
  \end{subfigure}\quad
  \begin{subfigure}[b]{0.2\textwidth} \centering
\begin{tikzpicture}[thick,level distance=10mm]
\tikzstyle{level 1}=[sibling distance=11mm]
\tikzstyle{level 2}=[sibling distance=8mm]
\tikzstyle{level 3}=[sibling distance=6mm]
\node [circle,draw]{}
child {child {child {[fill] circle (2pt)} child {[fill] circle (2pt)}} child {[fill] circle (2pt)}}
child {[fill] circle (2pt)};
\end{tikzpicture}
   \caption{$4$}
  \end{subfigure}\quad
 \begin{subfigure}[b]{0.2\textwidth} \centering
  \begin{tikzpicture}[thick,level distance=10mm]
\tikzstyle{level 1}=[sibling distance=10mm]
\tikzstyle{level 2}=[sibling distance=7mm]
\node [circle,draw]{}
child {child {[fill] circle (2pt)}child {[fill] circle (2pt)}}
child {[fill] circle (2pt)}
child {[fill] circle (2pt)};
\end{tikzpicture}
  \caption{$4$}
  \end{subfigure}\quad
  \begin{subfigure}[b]{0.2\textwidth} \centering
 \begin{tikzpicture}[thick,level distance=11mm]
\tikzstyle{level 1}=[sibling distance=7mm]
\node [circle,draw]{}
child {[fill] circle (2pt)}
child {[fill] circle (2pt)}
child {[fill] circle (2pt)}
child {[fill] circle (2pt)};
\end{tikzpicture}
  \caption{$4$}
  \end{subfigure}
\caption{All the topological trees with fewer than five leaves.}\label{fewer5leaves}
\end{figure}
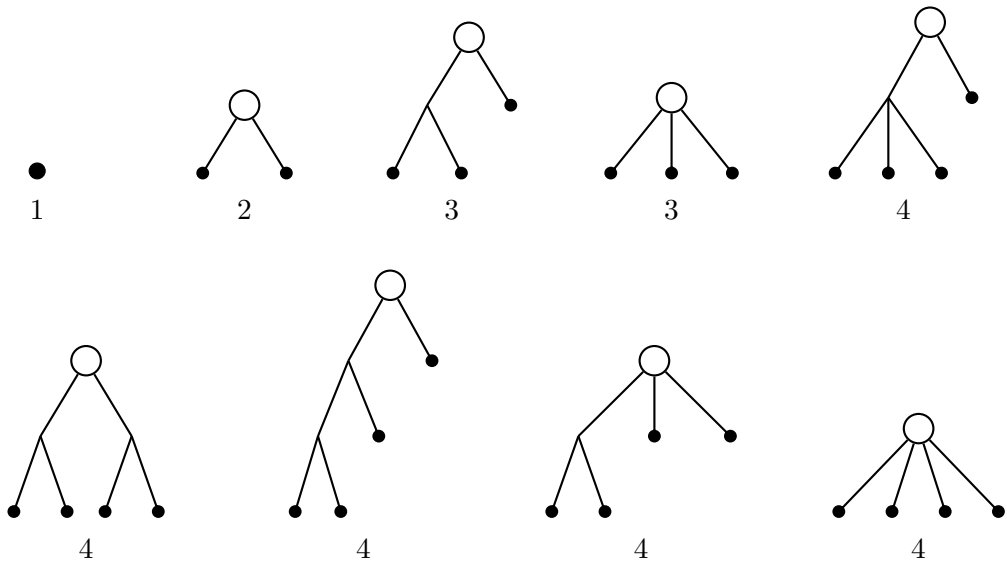

Every subset $L$ of leaves of a topological tree $T$ induces another tree which is obtained by the following operation: first extract the smallest subtree of $T$ that connects the leaves in $L$ and then erase the vertices of outdegree $1$ (if any). This operation is illustrated in Figure~\ref{leaf-induced}.

\begin{figure}[!h]\centering  
\begin{tikzpicture}[thick]
\node [circle,draw] (r) at (0,0) {};

\draw (r) -- (-2,-2);
\draw (r) -- (0,-4);
\draw (r) -- (2,-2);
\draw (-2,-2) -- (-2.5,-4);
\draw (-2,-2) -- (-1.5,-4);
\draw (2,-2) -- (1,-4);
\draw (2,-2) -- (2,-4);
\draw (2,-2) -- (3,-4);
\draw (1.25,-3.5) -- (1.5,-4);

\node [fill,circle, inner sep = 2pt ] at (-2.5,-4) {};
\node [fill,circle, inner sep = 2pt ] at (-1.5,-4) {};
\node [fill,circle, inner sep = 2pt ] at (0,-4) {};
\node [fill,circle, inner sep = 2pt ] at (1,-4) {};
\node [fill,circle, inner sep = 2pt ] at (1.5,-4) {};
\node [fill,circle, inner sep = 2pt ] at (2,-4) {};
\node [fill,circle, inner sep = 2pt ] at (3,-4) {};

\node at (-1.5,-4.5) {$\ell_1$};
\node at (0,-4.5) {$\ell_2$};
\node at (1,-4.5) {$\ell_3$};
\node at (2,-4.5) {$\ell_4$};

\node [circle,draw] (r1) at (6,-2) {};

\draw (r1) -- (5,-4);
\draw (r1) -- (6,-4);
\draw (r1) -- (7,-4);
\draw (6.75,-3.5)--(6.5,-4);

\node [fill,circle, inner sep = 2pt ] at (5,-4) {};
\node [fill,circle, inner sep = 2pt ] at (6,-4) {};
\node [fill,circle, inner sep = 2pt ] at (6.5,-4) {};
\node [fill,circle, inner sep = 2pt ] at (7,-4) {};

\node at (5,-4.5) {$\ell_1$};
\node at (6,-4.5) {$\ell_2$};
\node at (6.5,-4.5) {$\ell_3$};
\node at (7,-4.5) {$\ell_4$};

\end{tikzpicture}
\caption{A ternary tree and the subtree induced by four leaves $\{\ell_1,\ell_2,\ell_3,\ell_4\}$.}\label{leaf-induced}
\end{figure}
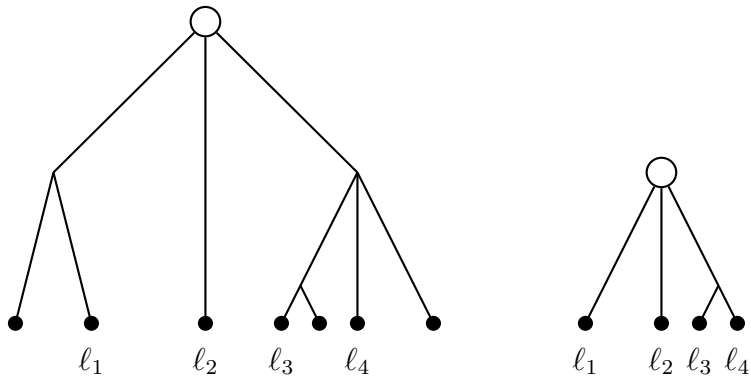

Such a subtree will be referred to as a \textit{leaf-induced subtree} of $T$. It has a root in a natural way which we define to be the most recent common parent shared by the leaves in $L$. We shall write $|T|$ for the number of leaves of a topological tree $T$.

\medskip
For two topological trees $S,T$, we denote by $c(S,T)$ the number of copies of $S$ in $T$. Formally, $c(S,T)$ is the number of subsets of the leaf set of $T$ that induce a tree isomorphic in the sense of rooted trees to $S$ (i.e., the isomorphism maps the root of one tree to the root of the other tree). Let $\gamma(S,T)=c(S,T)/\binom{|T|}{|S|}$ be its normalised version, which lies between $0$ and $1$ by definition. For brevity, $\gamma(S,T)$ will be referred to as the \textit{density} of $S$ in $T$. We are specifically interested in the maximum of $\gamma(S,T)$ in the limit: the quantity
\begin{align}\label{formuJS}
J(S):= \limsup_{\substack{|T| \to \infty \\ T~\text{topological tree}}} \gamma(S,T) = \limsup_{n \to \infty} \max_{\substack{|T| =n \\ T~\text{topological tree}}} \gamma(S,T)\,,
\end{align}
where the maximum runs over all topological trees, will be called the \textit{inducibility of $S$} (in topological trees).

\medskip
A topological tree with the property that every vertex has outdegree no more than $d~(\geq 2)$ will be called a \textit{$d$-ary tree}. In this note, we shall simply call a $2$-ary tree a \textit{binary} tree.

When the maximum of the density $\gamma(D,T)$ of a $d$-ary tree $D$ in $T$ is taken over all $d$-ary trees, we shall speak of the \textit{inducibility in $d$-ary trees}. Strictly speaking, the \textit{inducibility} $I_d(D)$ of a $d$-ary tree $D$ in $d$-ary trees is
\begin{align*}
I_d(D) = \lim_{n\to \infty} \max_{\substack{|T|=n \\ T~\text{$d$-ary tree}}} \gamma(D,T)\,,
\end{align*}
as presented in a previous article \cite{AudaceStephanPaper1}, where the limit is shown to exist. We will prove that this is also the case for $J(S)$, see Theorem~\ref{RelationMaxJS}.

\medskip
The $d$-ary trees that attain the maximum inducibility $1$ are determined in \cite{AudaceStephanPaper1}: for every $d\geq 2$, the maximal trees are so-called binary caterpillars (paths with one pendant edge dropped from all the vertices except for one of the endvertices). Things change, however, when the vertices of the tree are allowed to have any outdegree, see Theorem~\ref{MaxTrees}.

\medskip
Clearly, for any given $d$-ary tree $D$, the inducibility $I_d(D)$ is non-decreasing with respect to $d$, and thus tends to a definite limit as $d \to \infty$. In the following section, we establish that $J(S)=\lim_{d\to \infty} I_d(S)$ for every topological tree $S$. We also prove a link between the maximum density of $S$ in topological trees and the actual inducibility $J(S)$, and demonstrate that stars and binary caterpillars are the only topological trees with the maximum inducibility $1$.

Furthermore, we find an explicit lower bound on the limit inferior of the proportion of all subsets of leaves of $T$ that induce either a star or a binary caterpillar as the size of the topological tree $T$ grows to infinity (see Proposition~\ref{infTopol}). In addition, we show that one can obtain a lower bound on the inducibility of any topological tree.

\section{Auxiliary results}

We recall the following fundamental results from \cite{AudaceStephanPaper1}. We shall utilise them in various places in the note.

\begin{theorem}[\cite{AudaceStephanPaper1}]\label{maxdensityId}
For every fixed positive integer $d\geq 2$ and every $d$-ary tree $D$, the double inequality
\begin{align*}
0\leq \max_{\substack{|T|=n\\T~\text{$d$-ary tree}}}  \gamma(D,T) - I_d(D) \leq |D|(-1+|D|)n^{-1}
\end{align*}
holds for all $n$.
\end{theorem}

\begin{theorem}[\cite{AudaceStephanPaper1}]\label{max1dAry}
Let $d\geq 2$ be an arbitrary but fixed positive integer. Among $d$-ary trees, only binary caterpillars have ($d$-ary) inducibility $1$.
\end{theorem}

\section{Main results}

As for the degree-restricted inducibility, we begin our investigation by providing an estimate on how much the general inducibility can differ from the maximum density $\gamma(S,T)$ in topological trees.

The result is an analogue of Theorem~\ref{maxdensityId} for $J(S)$. It asserts that for every topological tree $S$, the maximum of $\gamma(S,T)$ tends to a definite limit as $|T| \to \infty$ and that the precise gap between the maximum density and the limit is of order at most $\mathcal{O}(|T|^{-1})$.

\begin{theorem}\label{RelationMaxJS}
Let $S$ be a topological tree. Then  the double inequality
\begin{align*}
0\leq \max_{\substack{|T|=n \\ T~\text{topological tree}}}  \gamma(S,T) - J(S) \leq |S|(-1+|S|)n^{-1}
\end{align*}
is valid for all $n \geq |S|$.  In particular, we have
\begin{align*}
J(S)=\lim_{n\to \infty}\max_{\substack{|T|=n \\ T~\text{topological tree}}} \gamma(S,T)\,.
\end{align*}
\end{theorem}

\begin{proof}
We can follow the same averaging argument used in \cite{AudaceStephanPaper1} to prove Theorem~\ref{maxdensityId}. Since we shall make heavy use of the intermediary results that appear in the proof, we present them for completeness.

Let $S$ and $T$ be two topological trees such that $|S|\leq |T|$. We write $L(T)$ for the set of leaves of $T$. For $l \in L(T)$, let us denote by $c_{l}(S,T)$ the number of $l$-containing subsets of leaves of $T$ that induce a copy of $S$. Thus, since
\begin{align*}
\sum_{l \in L(T)} c_{l}(S,T) =|S|\cdot c(S,T)\,,
\end{align*}
we deduce that there exist leaves $l_1$ and $l_2$ of $T$ (possibly $l_1=l_2$) for which the double inequality
\begin{align}\label{InView}
c_{l_1}(S,T) \leq \frac{|S| \cdot c(S,T)}{|T|} \leq c_{l_2}(S,T)
\end{align}
is satisfied. The number of copies of $S$ in $T$ not involving the leaf $l_1$ is
\begin{align*}
c(S,T) - c_{l_1}(S,T)\geq \Big(1-\frac{|S|}{|T|} \Big)c(S,T)
\end{align*}
by relation~\eqref{InView}. Let $T^{-}$ be the topological tree that results when the leaf $l_1$ of $T$ is removed and the unique vertex adjacent to $l_1$ is erased if it has outdegree $2$ in $T$. We have
$$c(S,T^{-})\geq \Big(1-\frac{|S|}{|T|} \Big)c(S,T)\,,$$
and dividing by $\binom{|T|-1}{|S|}$, we obtain
$$\gamma(S,T^-) \geq \gamma(S,T).$$
Since $T$ is an arbitrary topological tree, this implies that
\begin{align*}
\max_{\substack{|T|=n-1\\T~\text{topological tree}}} \gamma(S,T) \geq \max_{\substack{|T|=n\\T~\text{topological tree}}} \gamma(S,T)
\end{align*}
for every $n\geq 2$. In particular, the assertion on the limit follows as the sequence
\begin{align*}
\Big( \max_{\substack{|T|=n\\T~\text{topological tree}}} \gamma(S,T)\Big)_{n\geq 1}
\end{align*}
is nonincreasing and bounded. Moreover, we get
\begin{align*}
J(S) \leq \max_{\substack{|T|=n\\T~\text{topological tree}}} \gamma(S,T)
\end{align*}
for all $n$.

\medskip
Now, denote by $T^{+}$ the tree obtained after replacing the leaf $l_2$ of $T$ by an internal vertex with two leaves $l_2$ and $l_2^{\prime}$ attached to it. So, $c(S,T)$ represents the number of copies of $S$ in $T^{+}$ not involving $l_2^{\prime}$ (removing the leaf $l_2^{\prime}$ yields $T$) whereas the number of copies of $S$ in $T^{+}$ involving $l_2^{\prime}$ is at least $c_{l_2}(S,T)$. Thus, it follows from relation~\eqref{InView} that
$$c(S,T^{+})\geq \Big(1+\frac{|S|}{|T|}\Big) c(S,T).$$
Dividing by $\binom{|T|+1}{|S|}$ yields
$$\gamma(S,T^{+}) \geq \Big( 1 - \frac{|S|(-1+|S|)}{|T|(|T|+1)} \Big) \gamma(S,T),$$
and since $T$ was assumed to be an arbitrary $n$-leaf topological tree, we have
$$
\max_{\substack{|T|={n+1} \\ T~\text{topological tree}}} \gamma(S,T) \geq  \Big( 1 - \frac{|S|(-1+|S|)}{n(n+1)} \Big)
\max_{\substack{|T|={n} \\ T~\text{topological tree}}} \gamma(S,T)
$$
for every $n$. After $p$ iterations, we establish that
\begin{align*}
\max_{\substack{|T|=n+p\\T~\text{topological tree}}} \gamma(S,T) \geq  \Big(\max_{\substack{|T|=n\\T~\text{topological tree}}}\gamma(S,T)\Big) \cdot \prod_{j=0}^{p-1}\Bigg(1- \frac{|S|(-1+|S|)}{(n+p-j)(n+p-j-1)}\Bigg)
\end{align*}
for all $n,p$ with $p\geq 1$.

Letting $p \to \infty$ gives us the estimate
\begin{align*}
J(S) &\geq \Big(\max_{\substack{|T|=n\\T~\text{topological tree}}} \gamma(S,T)\Big)  \cdot \Bigg(1-\sum_{i=0}^{\infty} \frac{|S|(-1+|S|) }{(n+i+1)(n+i)} \Bigg)\\
&=\Big( \max_{\substack{|T|=n\\T~\text{topological tree}}} \gamma(S,T)\Big) \cdot \Bigg(1-\frac{|S|(-1+|S|)}{n} \Bigg)
\end{align*}
for every $n$, where in the first step we used the standard inequality (which is proved by a simple induction) $\prod_{i=1}^m (1-y_i)\geq 1-\sum_{i=1}^m y_i$ (valid when $0\leq y_i\leq 1$ for every $i$ -- see~\cite[p. 60]{hardy1952inequalities}) giving us
\begin{align*}
\prod_{j=0}^{p-1}\Bigg(1- \frac{|S|(-1+|S|)}{(n+p-j)(n+p-j-1)}\Bigg) \geq 1-\sum_{j=0}^{p-1} \frac{|S|(-1+|S|) }{(n+p-j)(n+p-j-1)}
\end{align*}
(recall that $n\geq |S|$). Putting things together, we obtain the desired double inequality and in particular, the asymptotic formula 
\begin{align*}
\max_{\substack{|T|=n \\ T~\text{topological tree}}}  \gamma(S,T) =J(S) +\mathcal{O}(n^{-1})
\end{align*}
which is valid for all $n$. This completes the proof of the theorem.
\end{proof}

The problem of determining extremal graph structures that maximise or minimise a given graph parameter is a topical subject within graph theory. In what follows, we address the question of characterising all the topological trees that have the maximal inducibility $1$.

Unlike the degree-restricted inducibility $I_d(S)$, two families of topological trees are found to be maximal with respect to $J(S)$.

\medskip
By a \textit{binary caterpillar}, we mean a binary tree whose non-leaf vertices lie on a single path starting at the root. By a \textit{star}, we mean a topological tree in which all leaves are children of the root.

An illustration of both classes of trees is given in Figure~\ref{binaAndStars}.

\begin{figure}[htbp]\centering
\begin{subfigure}[b]{0.3\textwidth} \centering 
\begin{tikzpicture}[thick,level distance=18mm]
\tikzstyle{level 1}=[sibling distance=6mm]   
\node [circle,draw]{}
child {[fill] circle (2pt)}
child {[fill] circle (2pt)}
child {[fill] circle (2pt)}
child {[fill] circle (2pt)}
child {[fill] circle (2pt)}
child {[fill] circle (2pt)};
\end{tikzpicture}
\caption{The star with six leaves}
  \end{subfigure}\qquad \qquad
\begin{subfigure}[b]{0.4\textwidth} \centering  
\begin{tikzpicture}[thick,level distance=8mm]
\tikzstyle{level 1}=[sibling distance=12mm]   
\node [circle,draw]{}
child {[fill] circle (2pt)}
child {child {[fill] circle (2pt)}child {child {[fill] circle (2pt)}child {child {[fill] circle (2pt)}child {[fill] circle (2pt)}}}};
\end{tikzpicture}
\caption{The binary caterpillar with five leaves}
  \end{subfigure}
\caption{A star and a binary caterpillar}\label{binaAndStars}
\end{figure}
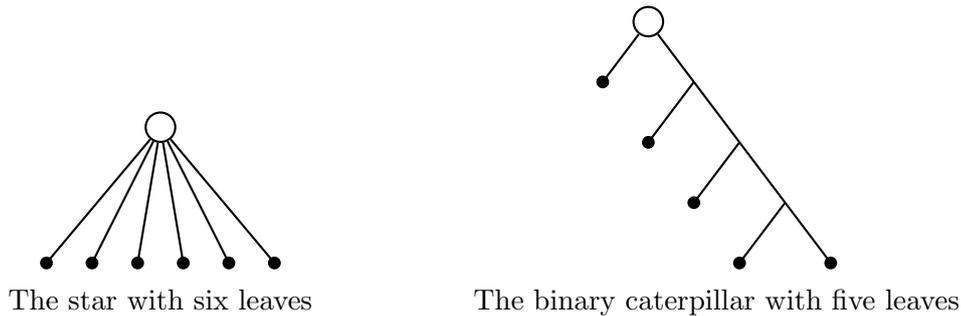

\begin{theorem}\label{MaxTrees}
Both stars and binary caterpillars have the maximal inducibility $1$. Moreover, every topological tree $S$ that is not a star or a binary caterpillar satisfies $J(S)<1$.
\end{theorem}

We remark that Theorem~\ref{MaxTrees} has a link to an observation made by Bubeck and Linial in \cite{bubeck2016local}. In their context, the inducibility $\text{ind}(R)$ of a tree $R$ (not necessarily a topological tree) with $k$ vertices is defined to be the limit superior of the proportion of $R$ as a subtree among all $k$-vertex induced subtrees where the limit is taken over all sequences of trees whose number of vertices grows to infinity.

In the last section of their paper, Bubeck and Linial point out that stars and paths are the only trees $R$ that have an inducibility $\text{ind}(R)$ of $1$.

\begin{proof}[Proof of Theorem~\ref{MaxTrees}]
First off, note that we have both $J(C_k)=1$ and $J(F^2_k)=1$ for every $k$ in view of the identity
\begin{align*}
 c(C_k,C_n)=c(F^2_k,F^2_n)=\binom{n}{k}\,,
\end{align*}
i.e., every subset of $k$ leaves of the star $C_n$ induces $C_k$, and every subset of $k$ leaves of the binary caterpillar $F^2_n$ induces $F^2_k$.

\medskip
The rest of the proof is a refinement of the proof of Theorem~\ref{max1dAry}. To begin, let us prove that for any fixed positive integers $k>2$ and $n>(k-1)^{k-2}$, every $n$-leaf topological tree contains either a copy of the $k$-leaf star $C_k$ or a copy of the $k$-leaf binary caterpillar $F^2_k$.

Note that a topological tree has $C_k$ as a leaf-induced subtree if and only if it contains a vertex that has at least $k$ children. Likewise, a topological tree contains $F^2_k$ as a leaf-induced subtree if and only if it has height at least $k-1$.

Fix $k>2$, $n>(k-1)^{k-2}$, and consider a topological tree $T$ with $n$ leaves. If we suppose that $T$ is a $(k-1)$-ary tree and $T$ has height at most $k-2$, then $T$ must have at most $(k-1)^{k-2}$ leaves because it is easy to see that for any fixed integers $h\geq 0$ and $m\geq 2$, an $m$-ary ($m\geq 2$) tree with height at most $h$ can never have more than $m^h$ leaves. This contradicts our assumption that $|T|=n>(k-1)^{k-2}$.

Therefore, for $k>2$ and $n>(k-1)^{k-2}$, every $n$-leaf topological tree must contain either a copy of the star $C_k$ or a copy of the binary caterpillar $F^2_k$.

\medskip
As a next step, if $S$ is neither $F^2_{|S|}$ nor $C_{|S|}$, then by the above discussion, we obtain $c(S,T)<\binom{|T|}{|S|}$ for every topological tree with $n$ leaves as soon as $n>(k-1)^{k-2}$. Therefore, fixing $n>(k-1)^{k-2}$, we get
\begin{align*}
\max_{\substack{|T|=n\\ T~\text{topological tree}}} \gamma(S,T) <1\,.
\end{align*}

Furthermore, we know from the proof of Theorem~\ref{RelationMaxJS} that
\begin{align*}
J(S)\leq \max_{\substack{|T|=n\\ T~\text{topological tree}}} \gamma(S,T)
\end{align*}
for every $n$: this finishes the proof of the theorem.
\end{proof}

In particular, one obtains:
\begin{corollary}
We have
\begin{align*}
\liminf_{\substack{|T| \to \infty \\ T~\text{topological tree}}} \gamma(S,T) =0
\end{align*}
for every topological tree $S$ with at least three leaves.
\end{corollary}

\begin{proof}
This is a consequence of the identities $J(C_k)=1$ and $J(F^2_k)=1$ because
\begin{align*}
\liminf_{\substack{|T| \to \infty \\ T~\text{topological tree}}} \gamma(S,T) \leq \liminf_{\substack{|T| \to \infty \\ T~\text{topological tree}}} \big(1- \gamma(F^2_{|S|},T)  \big) =1 - J(F^2_{|S|})
\end{align*}
for every topological tree $S$ different from a binary caterpillar, and likewise
\begin{align*}
\liminf_{\substack{|T| \to \infty \\ T~\text{topological tree}}} \gamma(S,T) \leq \liminf_{\substack{|T| \to \infty \\ T~\text{topological tree}}} \big(1- \gamma(C_{|S|},T) \big)=1-J(C_{|S|})
\end{align*}
for every topological tree $S$ different from a star.
\end{proof}

As one might expect (in view of the aforementioned analogy to the work~\cite{bubeck2016local}), the quantity
\begin{align*}
\liminf_{\substack{|T| \to \infty \\ T~\text{topological tree}}}\big(\gamma(C_k,T) +\gamma(F^2_k,T)\big)
\end{align*}
is positive for every $k$. It seems arduous to determine its explicit value as a function of $k$. Nevertheless, we are able to say more about this quantity.

\medskip
In fact, Bubeck and Linial~\cite{bubeck2016local} proved that in the context of induced subtrees of trees, the sum of the proportions of the $k$-vertex path and the $k$-vertex star is always greater than zero in the limit for every $k$. Specifically, they showed that the sum of the two proportions is bounded from below by an explicit constant that depends solely on $k$.

\begin{proposition}\label{infTopol}
For every fixed positive integer $k \geq 2$, the inequalities
\begin{align*}
\liminf_{\substack{|T| \to \infty \\ T~\text{topological tree}}}\big(\gamma(C_k,T) +\gamma(F^2_k,T)\big) \geq \frac{1}{\dbinom{1+(k-1)^{k-2}}{k}} \geq \frac{k!}{\Big( 1+(k-1)^{k-2}\Big)^k}
\end{align*}
are satisfied.
\end{proposition}

\begin{proof}
Let $k\geq 2$ be an arbitrary but fixed positive integer. In analogy to the proof of Theorem~\ref{RelationMaxJS}, we have
\begin{align*}
\sum_{l \in L(T)}  \big(c_l(F^2_k,T)+c_l(C_k,T) \big) =k\big( c(F^2_k,T)+c(C_k,T)\big)
\end{align*}
for every topological tree $T$. So there exists $l_1 \in L(T)$ such that the inequality
\begin{align*}
c_{l_1}(F^2_k,T)+c_{l_1}(C_k,T) \geq \frac{k}{|T|} \cdot \big( c(F^2_k,T)+c(C_k,T)\big)
\end{align*}
holds. If we denote by $T^{-}$ the topological tree obtained when removing the leaf $l_1$ of $T$ as well as erasing the unique vertex adjacent to $l_1$ in $T$ if it has outdegree $2$ in $T$, then this implies that
\begin{align*}
c(F^2_k,T^{-})+c(C_k,T^{-}) &=\big( c(F^2_k,T)+c(C_k,T)\big)- \big( c_{l_1}(F^2_k,T)+c_{l_1}(C_k,T)\big)\\
& \leq \Big(1-\frac{k}{|T|} \Big)\cdot \big( c(F^2_k,T)+c(C_k,T)\big),
\end{align*}
and dividing by $\binom{|T|-1}{k}$ yields
$$\gamma(F^2_k,T^{-})+\gamma(C_k,T^{-}) \leq \gamma(F^2_k,T)+\gamma(C_k,T)$$
for every topological tree $T$ with $n\geq 2$ leaves.
\medskip
Since $T$ is arbitrary, it follows that
$$\min_{\substack{|T|=n-1\\ T~\text{topological tree}}} (\gamma(F^2_k,T)+\gamma(C_k,T)) \leq \min_{\substack{|T|=n\\ T~\text{topological tree}}} (\gamma(F^2_k,T)+\gamma(C_k,T)),$$
so the sequence 
\begin{align*}
\Bigg(\min_{\substack{|T|=n\\ T~\text{topological tree}}}\big(\gamma(C_k,T) +\gamma(F^2_k,T) \big)\Bigg)_{n\geq k}
\end{align*}
is nondecreasing. Using this result, we derive that
\begin{align*}
\min_{\substack{|T|=n\\ T~\text{topological tree}}}\Big(\gamma(C_k,T) +\gamma(F^2_k,T) \Big) \geq \min_{\substack{|T|=1+(k-1)^{k-2}\\ T~\text{topological tree}}}\Big(\gamma(C_k,T) +\gamma(F^2_k,T) \Big) 
\end{align*}
for every $n>(k-1)^{k-2}$. 

It follows that
\begin{align*}
\liminf_{\substack{|T| \to \infty \\ T~\text{topological tree}}}\Big(\gamma(C_k,T) +\gamma\big(F^2_k,T\big)\Big) \geq \min_{\substack{|T|=1+(k-1)^{k-2}\\ T~\text{topological tree}}}\Big(\gamma(C_k,T) +\gamma(F^2_k,T)\Big) \,.
\end{align*}

Furthermore, we know from the proof of Theorem~\ref{MaxTrees} that
\begin{align*}
\gamma(C_k,T) +\gamma(F^2_k,T) \geq \frac{1}{\dbinom{|T|}{k}}
\end{align*}
as soon as $|T|> (k-1)^{k-2}$. Hence, we obtain
\begin{align*}
\liminf_{\substack{|T| \to \infty \\ T~\text{topological tree}}}\big(\gamma(C_k,T) +\gamma(F^2_k,T)\big)\geq \frac{1}{\dbinom{1+(k-1)^{k-2}}{k}}\,.
\end{align*}
Consequently, this finishes the proof of the proposition.
\end{proof}

We would very much welcome seeing a solution to the following question:

\begin{question}\label{QuesInfTopo}
What is the precise value of
\begin{align*}
\liminf_{\substack{|T| \to \infty \\ T~\text{topological tree}}}\big(\gamma(C_k,T) +\gamma(F^2_k,T)\big)
\end{align*}
when $k>3$?
\end{question}

Our next result gives a simple identity involving $J(S)$ and its $d$-ary counterpart $I_d(S)$. Specifically, although using equation~\eqref{formuJS} can be very difficult, we discover that to compute the inducibility $J(S)$, it is enough to properly understand the inducibility $I_d(S)$ as a single variable function of $d$.

\begin{theorem}\label{JSlimId}
We have 
\begin{align*}
J(S)= \lim_{d \to \infty} I_d(S)
\end{align*}
for every topological tree $S$.
\end{theorem}

\begin{proof}
Let $\epsilon > 0 $ be an arbitrary but fixed positive real number. We shall make use of the two estimates
\begin{equation}\label{J_upper_est}
J(S) \leq \max_{\substack{|T|=n \\ T~\text{topological tree}}}  \gamma(S,T)
\end{equation}
and
\begin{align*}
I_d(S) \geq \Bigg(1- \frac{|S|(-1+|S|)}{n}\Bigg)\max_{\substack{|T|=n \\ T~d-\text{ary tree}}}  \gamma(S,T)
\end{align*}
both valid for every $n$. They follow from the proof of Theorem~\ref{RelationMaxJS} and Theorem~\ref{maxdensityId}, respectively.

For any positive integer $\displaystyle m_{\epsilon} > |S|(-1+|S|)/\epsilon $, we have
\begin{align*}
\frac{|S|(-1+|S|)}{m_{\epsilon}} \cdot  \max_{\substack{|T|=m_{\epsilon} \\ T~m_{\epsilon}-\text{ary tree}}} \gamma(S,T) & < \epsilon\,,
\end{align*}
which implies
\begin{align}\label{IneBecom}
\max_{\substack{|T|=m_{\epsilon} \\ T~\text{topological tree}}} \gamma(S,T) -\Bigg(1- \frac{|S|(-1+|S|)}{m_{\epsilon}}\Bigg)\max_{\substack{|T|=m_{\epsilon} \\ T~m_{\epsilon}-\text{ary tree}}} \gamma(S,T) & < \epsilon\,.
\end{align}

Employing the relation
\begin{align*}
I_{m_{\epsilon}}(S) \geq \Bigg(1- \frac{|S|(-1+|S|)}{m_{\epsilon}}\Bigg)\max_{\substack{|T|=m_{\epsilon} \\ T~m_{\epsilon}-\text{ary tree}}}  \gamma(S,T)\,,
\end{align*}
and invoking the fact that $\big(I_m(S)\big)_{m\geq 2}$ is a nondecreasing sequence of real numbers, inequality~\eqref{IneBecom} implies
\begin{align*}
\Big(\max_{\substack{|T|=m_{\epsilon} \\ T~\text{topological tree}}} \gamma(S,T) \Big) - I_m(S)  < \epsilon\,,
\end{align*}
for every $m \geq m_{\epsilon}$.

\medskip
Combining this with~\eqref{J_upper_est}, we establish that
\begin{align*}
J(S) -I_m(S)  < \epsilon
\end{align*}
for every $m \geq m_{\epsilon}$. On the other hand, we have
\begin{align*}
J(S) \geq \sup_{d\geq 2} I_d(S) = \lim_{d\to \infty} I_d(S)
\end{align*}
by definition of $J(S)$. Hence, we conclude that for every $\epsilon >0$, there exists $\Delta \geq 2$ such that the inequality
\begin{align*}
|J(S) -I_d(S)|  < \epsilon 
\end{align*}
holds for every $d>\Delta$. This completes the proof of the theorem.
\end{proof}

\begin{remark}
For a topological tree $T$, let us denote by $\displaystyle \Delta(T)$ the maximum among the outdegrees of the vertices of $T$. Let $T_1^{\bullet},T_2^{\bullet},\ldots$ be a sequence of topological trees such that $|T_n^{\bullet}| \to \infty$ as $n \to \infty$ and 
\begin{align*}
J(S)=\lim_{n \to \infty} \gamma(S,T_n^{\bullet})\,.
\end{align*}

If $\displaystyle \max_{i\geq 1}\Delta(T_i^{\bullet})$ exists, then the sequence $\big(I_d(S)\big)_{d\geq 2}$ becomes constant from the point $\displaystyle \Delta=\max_{i\geq 1}\Delta(T_i^{\bullet})$ onwards, that is,
\begin{align*}
I_{\Delta}(S)=I_{1+\Delta}(S)=I_{2+\Delta}(S)=\cdots \,.
\end{align*}
\end{remark}

\medskip
As already mentioned earlier, Theorem~\ref{JSlimId} can be very useful. The following corollary reveals one of its features:

\begin{corollary}\label{corIBJB}
We have
\begin{align*}
J(B)=I_2(B) 
\end{align*}
for every binary tree $B$.
\end{corollary}

\begin{proof}
It is shown in \cite{AudaceStephanPaper1} that $I_d(B)=I_2(B)$ for every $d$ and binary tree $B$. Therefore, by Theorem~\ref{JSlimId}, we obtain $J(B)=I_2(B)$.
 \end{proof}
 
\medskip
For the next corollary, we first need to recall the definition of a complete $d$-ary tree.

The \textit{complete $d$-ary tree} $CD^d_h$ of height $h$ is the $d$-ary tree defined recursively in the following way:
\begin{itemize}
\item $CD^d_0$ is the single leaf; 
\item for $h>0$, $CD^d_h$ has $d$ branches isomorphic to the tree $CD^d_{h-1}$.
\end{itemize}

\begin{corollary}
There are infinitely many topological trees $S$ with $J(S) \leq \epsilon$ for every $\epsilon >0$.
\end{corollary}

\begin{proof}
It is sufficient to consider binary trees. By a result that appears in~\cite{DossouOloryWagner}, we know that the inducibility $I_2(CD^2_h)$ of the complete binary tree $CD^2_h$ of height $h$ is at most $I_2(CD^2_2)^{2^{h-2}}$ for all $h\geq 2$. Therefore, since it is also proved in the same paper~\cite{DossouOloryWagner} (see also~\cite{czabarka2016inducibility}) that $I_2(CD^2_2)=3/7$, we deduce that
\begin{align*}
I_2(CD^2_h)\leq I_2(CD^2_2)^{2^{h-2}}=\Big(\frac{3}{7}\Big)^{2^{h-2}} \to 0~\text{as}~h \to \infty\,.
\end{align*}
Thus, the assertion of the corollary follows from Corollary~\ref{corIBJB}.
\end{proof}

We close this section by collecting two general facts about the inducibility $J(S)$. We note that some results in \cite{AudaceStephanPaper1} with respect to $I_d(S)$ hold analogously for $J(S)$ using our Theorem~\ref{JSlimId}.

\medskip
Note that $J(S)> 0$ for every topological tree $S$ because by definition, we have $J(S) \geq I_{\Delta (S)}(S)$ where $\Delta (S)$ is the maximum outdegree among the vertices of $S$ while it is proved in \cite{AudaceStephanPaper1} that $I_{\Delta (S)}(S)>0$. The following proposition provides a better lower bound that only depends on $|S|$.

\begin{proposition}\label{howsamllJS}
We have
\begin{align*}
J(S) \geq \frac{(-1+|S|)!}{-1+|S|^{-1+|S|}}
\end{align*}
for every topological tree $S$ with at least two leaves.
\end{proposition}

\begin{proof}
Consider a topological tree $S$. By a result in \cite{AudaceStephanPaper1}, we have 
\begin{align*}
I_d(S) \geq \frac{(-1+|S|)!}{-1+|S|^{-1+|S|}}
\end{align*}
for every $d \geq \Delta (S)$. Passing to the limit as $d\to \infty$, Theorem~\ref{JSlimId} gives us the desired statement.
\end{proof}

For two topological trees $S_1,S_2$, denote by $\mathcal{F}(S_1;S_2)$ the unique topological tree which is constructed by appending the root of $S_2$ to every leaf of $S_1$. See Figure~\ref{TreeFS1S2} for a diagram.

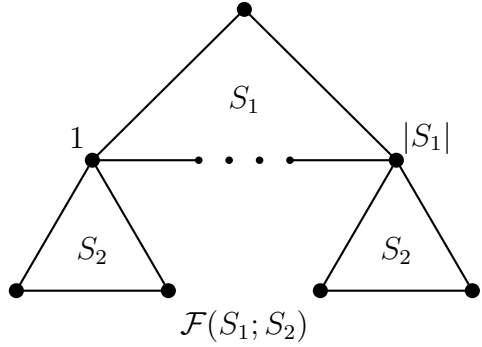
\begin{figure}[htbp]\centering
\begin{tikzpicture}[thick]

\node[fill=black,circle,inner sep=2pt] at (0,0) {};
\node[fill=black,circle,inner sep=2pt] at (4,0) {};
\node[fill=black,circle,inner sep=2pt] at (2,2) {};

\node[fill=black,circle,inner sep=2pt] at (-1,-1.73) {};
\node[fill=black,circle,inner sep=2pt] at (1,-1.73) {};
\node[fill=black,circle,inner sep=2pt] at (3,-1.73) {};
\node[fill=black,circle,inner sep=2pt] at (5,-1.73) {};

\node[fill=black,circle,inner sep=1pt] at (1.4,0) {};
\node[fill=black,circle,inner sep=1pt] at (1.8,0) {};
\node[fill=black,circle,inner sep=1pt] at (2.2,0) {};
\node[fill=black,circle,inner sep=1pt] at (2.6,0) {};

\node at (2,0.8) {$S_1$};
\node at (0,-1.2) {$S_2$};
\node at (4,-1.2) {$S_2$};

\draw (1.4,0)--(0,0)--(2,2)--(4,0)--(2.6,0);

\draw (-1,-1.73)--(1,-1.73)--(0,0)--cycle;
\draw (3,-1.73)--(5,-1.73)--(4,0)--cycle;

\node at (2,-2.2) {$\mathcal{F}(S_1;S_2)$};

\node at (-0.2,0.3) {$1$};
\node at (4.4,0.3) {$|S_1|$};

\end{tikzpicture}
\caption{A rough picture of the tree $\mathcal{F}(S_1;S_2)$ defined for Theorem~\ref{Thmhowsmal}.} \label{TreeFS1S2}
\end{figure}

\begin{theorem}\label{Thmhowsmal}
The tree $\mathcal{F}(S_1;S_2)$ satisfies
\begin{align*}
J\big(\mathcal{F}(S_1;S_2) \big) \geq \frac{(|S_1|\cdot |S_2|)!}{(|S_2|!)^{|S_1|}\cdot |S_1|^{|S_1|\cdot |S_2|}} \cdot J(S_2)^{|S_1|}
\end{align*}
for every pair of topological trees $S_1$ and $S_2$.
\end{theorem}

\begin{proof}
Again, by a result in \cite{AudaceStephanPaper1}, we have 
\begin{align*}
I_d\big(\mathcal{F}(S_1;S_2) \big) \geq \frac{(|S_1|\cdot |S_2|)!}{(|S_2|!)^{|S_1|}\cdot |S_1|^{|S_1|\cdot |S_2|}} \cdot I_d(S_2)^{|S_1|}
\end{align*}
for every $d \geq \Delta \big(\mathcal{F}(S_1;S_2) \big)$. Taking the limit of both sides as $d\to \infty$, and invoking Theorem~\ref{JSlimId}, we obtain the desired inequality of the theorem.
\end{proof}

It turns out that if one knows the inducibility of a topological tree $S$, then one can obtain a lower bound on the inducibility of any of the leaf-induced subtrees of $S$:

\begin{theorem}\label{RST}
For any three topological trees $R$, $S$ and $T$ such that $|T|\geq |S| \geq |R|$, we have
\begin{align*}
c(R,T)\geq \frac{c(R,S)}{\dbinom{-|R|+|T|}{-|R|+|S|}} \cdot c(S,T)\,. 
\end{align*}
In particular, we obtain 
\begin{align*}
J(R)\geq J(S) \cdot \gamma(R,S)\,.
\end{align*}
\end{theorem}

\begin{proof}
Given a topological tree $S$, we can count the number of appearances of any smaller topological tree $R$ in a larger topological tree $T$ by first counting the number of copies of $S$ in $T$ and then the number of copies of $R$ in $S$. Clearly, $c(R,T)$ is overcounted in this way based on the observation that the intersection of the subsets of leaves of $T$ that induce a copy of $S$ may not be empty. So we would like to take this observation into account.

Assume $|T|\geq |S| \geq |R|$. Given a subset $L$ of leaves of $T$ that induces a copy of $R$, we can then choose any subset $L^{\prime}$ of $|S|-|R|$ leaves from the set of leaves of $T$ without $L$ so that $L \cup L^{\prime}$ induces a subtree of $T$ with $|S|$ leaves. Since the subtree induced by the subset $L \cup L^{\prime}$ of leaves of $T$ has $|S|$ leaves, we deduce that the quantity $c(S,T)\cdot c(R,S)$ is at most $\binom{|T|-|R|}{|S|-|R|} \cdot c(R,T)$. Consequently, one obtains a simple lower bound on $c(R,T)$, namely
\begin{align*}
c(R,T) \geq \frac{c(S,T) \cdot c(R,S)}{\dbinom{|T|-|R|}{|S|-|R|}}\,.
\end{align*}

As a next step, we take the density:
\begin{align*}
\gamma(R,T) &\geq \gamma(R,S) \cdot \frac{\dbinom{|S|}{|R|}\cdot \dbinom{|T|}{|S|}}{\dbinom{|T|}{|R|} \cdot \dbinom{|T|-|R|}{|S|-|R|} }\cdot \gamma(S,T)\\
& =\gamma(R,S) \cdot \gamma(S,T)\,.
\end{align*}

Finally, in view of Theorem~\ref{RelationMaxJS}, we can take the limit to obtain the desired result.
\end{proof}

\begin{remark}
Let $d\geq 2$ be a fixed positive integer. By the same argument as in the proof of Theorem~\ref{RST}, it is also seen that
\begin{align*}
I_d(D_1)\geq I_d(D_2) \cdot \gamma(D_1,D_2)
\end{align*}
for any two $d$-ary tree $D_1$ and $D_2$ satisfying $|D_1|\leq |D_2|$.
\end{remark}

\section{Concluding comments}

It would be interesting to answer the following question:

\begin{question}
Can we explicitly determine the inducibility $J(S)$ of any topological tree $S$ other than a star or a binary tree? Note that in the degree-restricted context, the answer to this question is affirmative \cite{DossouOloryWagner}.
\end{question}

Finally, we conjecture that equality never holds in Proposition~\ref{howsamllJS}. It is also natural to formulate the following problem:

\begin{question}
Is there a sequence $(T_n)_{n\geq 1}$ of topological trees such that $|T_n|=n$ and $\lim_{n\to \infty} \gamma(S,T_n)$ exists and is positive for all topological trees $S$? In the degree-restricted context, the answer is positive \cite{DossouOloryWagner}.
\end{question}

\end{document}